\documentclass[a4paper]{article}
\usepackage{amsfonts}
\usepackage{amssymb}
\usepackage{amsmath}
\usepackage{amsthm}
\usepackage{graphicx}
\usepackage{subcaption}
\usepackage{multicol}
\usepackage{siunitx}
\usepackage{url}
\usepackage[utf8]{inputenc}
\usepackage[nottoc]{tocbibind}
\usepackage{tikz}
\usepackage[symbol]{footmisc}
\usepackage{xcolor}
\usepackage{mleftright}
\usepackage{verbatim}
\usepackage[ruled,vlined]{algorithm2e}
\usepackage[noend]{algpseudocode}
\usepackage{lipsum}
\usepackage{mathrsfs}
\usepackage{tikz-cd}
\usepackage{dsfont}
\usepackage{hyperref}
\hypersetup{
	colorlinks=true,
	linkcolor=blue,
	filecolor=magenta,      
	urlcolor=cyan,
}
\usepackage{tabularx}
\usepackage{multirow}
\usepackage{longtable}
\usepackage{emptypage}
\usepackage{cleveref}
\usepackage{enumitem}
\usepackage[toc,page]{appendix}

\usepackage{a4wide}

\newtheorem{lemma}{Lemma}[section]
\newtheorem{theorem}[lemma]{Theorem}
\newtheorem*{theorem*}{Theorem}
\newtheorem{definition}[lemma]{Definition}
\newtheorem{proposition}[lemma]{Proposition}
\newtheorem*{proposition*}{Proposition}

\newtheorem{corollary}[lemma]{Corollary}

\newtheorem{question}{Question}
\newtheorem{result}{Theorem}

\crefalias{result}{theorem}

\DeclareMathOperator{\alt}{Alt}
\DeclareMathOperator{\sym}{Sym}

\newcommand{\RR}{\mathbb{R}}

\newcommand{\ZZ}{\mathbb{Z}}

\newcommand{\NN}{\mathbb{N}}

\newcommand{\inv}{^{-1}}
\newcommand{\comp}{\circ}

\newcommand{\abs}[1]{\left|#1\right|}
\newcommand{\norm}[1]{\left|\left|#1\right|\right|}
\newcommand{\set}[1]{\left\{#1\right\}}

\newcommand{\normal}{\vartriangleleft}
\newcommand{\conj}{\mathrm{Conj}}
\newcommand{\rf}{\mathrm{Rf}}

\newcommand{\betha}{\beta}

\title{Arbitrary residual finiteness and conjugacy separability growth}
\author{Lukas Vandeputte\footnote{The author  kindly acknowledges the support by the group of science Engineering and Technology at KU Leuven Campus Kulak.}}
\begin{document}
	\maketitle
\begin{abstract}

	In a recent paper, Henry Bradford showed that all sufficiently fast growing functions appear as the residual finiteness growth function of some group. In this paper we show that the groups there constructed are conjugacy separable and that their conjugacy separability growth is equal to the residual finiteness growth. It follows that all sufficiently fast growing functions appear as the conjugacy separability growth function of some group. We extend this construction to a new class of groups such that given functions $f_1,f_2$ under the same constraints and satisfying $f_2\geq f_1$, we can find a group such that the residual finiteness growth is given by $f_1$ and the conjugacy separability growth by $f_2$, showing that the residual finiteness growth and conjugacy separability growth behave independently and can lie arbitrarily far apart.
\end{abstract}

	\section{Introduction}
	\sloppy
	In recent years there is an effort in measuring how well we can recognise properties of certain group elements in finite quotients. For this, Bou-Rabee introduced the residual finiteness growth function: $\rf_G^S:\NN\rightarrow\NN\cup\{\infty\}$, which, given a group $G$ with finite generating set $S$, maps an integer $n$ to the minimal integer $m$ (if it exists) such that for any non-trivial group element $g\in G$ of word norm $\norm{g}_S\leq n$, then there exists a finite quotient $\varphi:G\rightarrow Q$ satisfying $\abs{Q}\leq m$ and $\varphi(g)\neq 1$. Similarly \cite{Lawton2017decision} introduces the conjugacy separability growth function $\conj_G^S:\NN\rightarrow\NN\cup\{\infty\}$, which maps an integer $n$ to the minimal integer $m$ such that for $g_1,g_2\in G$ two non-conjugate elements satisfying $\norm{g_1}_S,\norm{g_2}_S\leq n$, there exists a finite quotient $\varphi:G\rightarrow Q$ such that $\abs{Q}\leq m$ and  such that $\varphi(g_1)$ and $\varphi(g_2)$ are non-conjugate.
	These functions only depend on the generating set up to a bounded factor in the argument. If two functions $f$ and $g$ only differ by such a factor in the argument, then we denote $f\simeq g$.
	If the group $G$ above is residually finite or conjugacy separable, then the function $\rf_G$ and $\conj_G$ respectively only take finite values. If a group $G$ is conjugacy separable, then it is also residually finite and we have $\conj_G^S\geq\rf_G^S$.
	
	A first way of studying these functions is by fixing a class of groups, and determining the behaviour of $\rf_G$ or $\conj_G$ for groups $G$ in that class, see \cite{dere2022survey} for a recent survey. We can however also pose the inverse question: For which functions $f$ can we find a group $G$ such that $\rf_G$ or $\conj_G\simeq f$. The first result in this manner was by Bou-Rabee and Steward \cite{Bou-Rabee} which demonstrates that we can have arbitrarily large residual finiteness growth. The groups they use are an adaptation of B.H. Neumann's continuous family of $2$-generator groups constructed in \cite{neumann1937some}. The authors however left ambiguous if these groups were in fact conjugacy separable and it thus remained unclear whether there exist groups of arbitrarily large conjugacy separability growth function. 
	
	A second result in this regard by Kharlampovich, Myasnikov and Sapir \cite{Kharlampovich}, demonstrates that the groups above can be chosen to be $3$-step solvable, and can have word-problem of a predefined complexity.
	
	More recently, it was shown by Bradford \cite{Bradford} that the class of functions appearing as $\rf_G$ is not only unbounded, but any function that grows sufficiently fast can appear (up to $\simeq$) as a residual finiteness growth function.
	More precisely, if $F:\NN\rightarrow\NN$ is a non-decreasing functions such that
	\begin{enumerate}[label=(\Alph*)]
		\item There exist $C,\epsilon>0$ such that for all $n\in \NN$,$$
		F(n)\geq \exp\left(Cn\log(n)^2\log\log(n)^{1+\epsilon}\right)\label{eq:globalgrowth}
		$$
		\item There exist $C_1,C_2>1$ such that for all $n\in\NN$,$$
		F(n)^{C_1}\leq F(C_2n)\label{eq:localgrowth}
		$$
	\end{enumerate}
	then there exists some group $G$ such that $\rf_G(n)\simeq F(n)$.
	
	The previous results demonstrate the spectrum of $\rf$ is relatively well understood (at least for large functions), however the literature is quiet about the spectrum of the conjugacy separability growth function.
	
	In this article we first demonstrate that the groups constructed in \cite{Bradford} are conjugacy separable, and that their conjugacy separability growth function is the same as their residual finiteness growth function. We thus obtain the following:	
	
	\begin{result}
		Let $F : \NN\rightarrow\NN$ be a non-decreasing function satisfying the conditions \ref{eq:globalgrowth} and \ref{eq:localgrowth}
		
		Then there exists a conjugacy separable group $G$, generated by a finite set $S$ such that  \begin{equation*}
		F(n)\simeq \conj_{G}^S(n).
		\end{equation*}
	\end{result}
	
	Although these groups have the same residual finiteness growth and conjugacy separability growth, for general groups this is not the case. However, previous work does not specify how far these functions can lie apart. We construct a new class of groups demonstrating that the residual finiteness growth and conjugacy separability growth may be chosen independently. More specifically, we have the following result.

	\begin{result}
		Let $F_1,F_2:\NN\rightarrow\NN$ be non-decreasing functions satisfying \ref{eq:globalgrowth} and \ref{eq:localgrowth} such that $F_2\geq F_1.$ Then there exists a conjugacy separable group $G$ such that \begin{itemize}
			\item $\rf_G(n)\simeq F_1(n)$;\\
			\item $\conj_G(n)\simeq F_2(n)$.\\
		\end{itemize}
	\end{result}
	
	We will start this paper by giving some limited background on effective separability, then in \cref{sec:bradconstruct}, we will give a construction of the groups from \cite{Bradford}. Our construction will be slightly different form the original construction but the groups obtained are isomorphic in a natural way. In \cref{sec:conjOnBrad} we will describe the conjugacy separability of these groups in \cref{prop:exactupperbound}. In \cref{sec:constructingGroups} we will construct our new class of groups, and use them to demonstrate \cref{prop:conjGap}.
	
	\section{Preliminaries}
	\subsection{Notations}
	Let $\NN$ be the set $\{1,2,3,\cdots\}$.
	Let $G$ be a group, we say $g_1,g_2\in G$ are conjugate if there exists some $h\in G$ such that $hg_1h\inv=g_2$. In this case we denote $g_1\sim g_2$ and otherwise we denote $g_1\not\sim g_2$.
	We denote the commutator ${[g_1,g_2]=g_1g_2g_1\inv g_2\inv}$.
	Let $N$ be a subgroup of a group $G$, then denote $[G:N]$ its index.
	
	Let $X$ be a finite set, then we denote the alternating group $\alt(X)$ the group of even permutations on $X$, acting from the left. If $n\in \NN$ is an integer, then we denote $\alt(n)=\alt\set{0,1,\cdots,n-1}$.
	Similarly, denote $\sym(X)$ the group of permutations on $X$, acting from the left and $\sym(n)=\sym\{0,1,\cdots, n-1\}$
		
	Let $F_{a,b}$ be the free group on the generators $a,b$, suppose that $w\in F_{a,b}$ and let $\alpha,\beta\in G$ be two elements of some group. There exists a unique morphism $\varphi_{\alpha,\beta}$ of $F_{a,b}$ to $G$, mapping $a$ to $\alpha$ and $b$ to $\beta$, denote with $w(\alpha,\beta)=\varphi_{\alpha,\beta}(w)$. 
	\subsection{Separability}
	A group $G$ is called \textbf{residually finite} if for any non-trivial element $g\in G$, there exists a finite quotient $\varphi:G\rightarrow Q$ such that $\varphi(g)$ is non-trivial. Related to this, a group $G$ is called \textbf{conjugacy separable} if for any pair of non-conjugate elements $g_1,g_2\in G$, there exists a finite quotient $\varphi:G\rightarrow Q$ such that $\varphi(g_1)$ and $\varphi(g_2)$ are non-conjugate. 
	By taking $g_1$ trivial, it follows that a conjugacy separable group is also residually finite.

	The ease of separating elements in finite quotients can then be measured through the \textbf{residual depth function} $\rf_G$ or the \textbf{conjugacy depth function} $\conj_G$ defined respectively as.
	$$
	\rf_G:G\rightarrow\NN:g\mapsto \begin{cases}
		\begin{matrix}
			1&\text{if } g=1\\
			\min\set{[G:N]\mid N\normal G, g\notin N}&\text{otherwise}
		\end{matrix}
	\end{cases}
	$$
	and
	$$
	\conj_{G}:G\times G\rightarrow \NN:(g_1,g_2)\mapsto\begin{cases}
		\begin{matrix}
			1&\text{if } g_1\sim g_2\\
			\min\set{[G:N]\mid N\normal G, g_1N\not\sim g_2N}&\text{otherwise}
		\end{matrix}
	\end{cases}
	$$
	Again we have $\conj_G(g,1)=\rf_G(g)$.
	
	Let $G$ be a group, finitely generated by $S$. We can then define the \textbf{word norm} $\norm{\_}_S:G\rightarrow \NN$ as follows:
	$$g\mapsto \begin{cases}\begin{matrix}
			0 &\text{if } g=e_G\\
			\min\set{n\in\NN\mid g=s_1s_2\cdots s_n\text{ where } s_i\in S\cup S\inv} &\text{otherwise}
		\end{matrix}
	\end{cases}$$

	If we compare the residual depth function or the conjugacy depth function with the word norm, then we obtain the \textbf{residual finiteness growth function} $\rf_G^S$ or the \textbf{conjugacy separability growth function} $\conj_G^S$ respectively, defined as:
	$$
	\rf_G^S:\NN\rightarrow\NN:n\mapsto \max\set{\rf_G(g)\mid g\in G,\norm{g}_S\leq n}
	$$
	and
	$$
	\conj_{G}^S:\NN\rightarrow \NN:n\mapsto \max\set{\conj_G(g_1,g_2)\mid g_1,g_2\in G,\norm{g_1}_S\leq n,\norm{g_2}\leq n}.
	$$
	
	These functions both depend on the choice of generating set. This dependence, however, is only slight. In particular we have:\begin{lemma}
		Let $G$ be a residually finite group and let $S$ and $T$ be finite generating sets of $G$. Then there exists a constant $C$ such that$$
		\rf_G^S(n)\leq\rf_G^T(Cn)
		$$
		and vice versa.
	\end{lemma}
	 The same holds for the function $\conj$ in conjugacy separable groups.
	\begin{proof}
		See for instance \cite[Lemma 1.1]{bou2010quantifying}.
	\end{proof}
	
	This motivates the following definition.
	\begin{definition}
		Let $f,g:\NN\rightarrow\RR_{>0}$ be non-decreasing functions. Then we say $f\prec g$ if there exists some natural constant $c$ such that $f(n)\leq g(cn)$ for all $n\in \NN_0$. Furthermore if $f\prec g$ and $g\prec f$, then we say $f\simeq g$.
	\footnote{In the literature,sometimes a variation of $\prec$ is used where $f(n)\prec g(n)$ if $f(n)\leq cg(cn)$. This is not equivalent in general, but thanks to condition\ref{eq:localgrowth} the main results remain the same under this alternative definition.}
	\end{definition}
	With notation as above, we then have for two finite generating sets $S$ and $T$ that $$
	\rf_G^S(n)\simeq \rf_G^T(n)
	$$
	and that 
	$$
	\conj_G^S(n)\simeq\conj_G^T(n).
	$$

	\section{Bradford's construction}\label{sec:bradconstruct}
	In this section we mention a first class of groups. These groups where constructed by Bradford in \cite{Bradford} as a generalisation to the B.H.Neumann groups from \cite{neumann1937some}. These groups where introduced to prove the following theorem:
	\begin{theorem}\cite[Theorem 1.1]{Bradford}\label{prop:mainBradResult}
		Let $F:\NN\rightarrow\NN$ be a non-decreasing function such that:
		\begin{enumerate}[label=(\Alph*)]
			\item There exist $c,\epsilon>0$ such that for all $n\in\NN$,$$
			F(n)\geq \exp(cn\log(n)^2\log\log(n)^{1+\epsilon})\label{cond:globalgrowth1};
			$$
			\item There exist $C_1,C_2>1$ such that for all $n\in\NN$,$$\label{cond:localgrowth1}
			F(n)^{C_1}\leq F(C_2(n)).
			$$
		\end{enumerate}
		Then there exists a residually finite group $G$, generated by a finite set S, and $C_1',C_2'>0$ such that for all $n\in\NN$ $$
		F(C_1'n)\leq \mathrm{Rf}_G^S\leq F(C_2'n).
		$$
	\end{theorem}

Let $F:\NN\rightarrow \NN$ be a non-decreasing function satisfying \ref{cond:globalgrowth1} and \ref{cond:localgrowth1},
\cite[Lemma 2.3.]{Bradford} states that for a given $C_1>0$, there exists some constant $C_2$ such that $$f(n)=\lceil\log F(n+C_2)/\log \log F(n+C_2) \rceil\geq \max(C_1n\log(n)\left(\log\log(n)\right)^{1+\frac{\epsilon}{2}},C_1)$$

We can use Bertrand's postulate to find some non-decreasing function $d:\NN\rightarrow \NN$, taking on prime values larger then $5$, such that $f(n)\leq d(n)\leq 2f(n)$.

The following lemma states that $d(n)!/2$ is a good approximation of $F(n)$

\begin{lemma}\label{prop:inverseFactorialAproximation}
	Let $F:\NN\rightarrow\NN>3$ be a non decreasing function satisfying \ref{cond:localgrowth1}. Let $g:\NN\rightarrow\NN$ be some non-decreasing function. Suppose that for $n$ sufficiently large $$\left\lceil\frac{\log F(n)}{\log\log F(n)}\right\rceil\leq g(n)\leq C_3\left\lceil\frac{\log F(n)}{\log\log F(n)}\right\rceil$$
	for some constant $C_3>1$, then $g(n)!/2\simeq F(n)$		
\end{lemma}
\begin{proof}
	First notice that $F(n)$ must tend to infinity.
	Let $f_0(n)=\left\lceil\frac{\log F(n)}{\log\log F(n)}\right\rceil$. 
	By \cite[Lemma 2.4.]{Bradford}, we have for $K\in \NN$ that $$
	(Kf_0)!=F(n)^{K+O_K(\frac{\log\log\log F(n)}{\log\log F(n)})}.
	$$
	Evaluating for $K=1$ and $K=C_3$ respectively, we obtain that for sufficiently large $n$,
	$$g(n)!\geq F(n)^{1-C_4\frac{\log\log\log F(n)}{\log\log F(n)}}$$
	and $$
	g(n)!\leq F(n)^{1+C_5\frac{\log\log\log F(n)}{\log\log F(n)}}
	$$
	for some constants $C_4,C_5>0$.
	Notice that $\frac{\log\log\log F(n)}{\log\log F(n)}$ tends to $0$ and thus for $n$ sufficiently large, we obtain 
	$$g(n)!\geq F(n)^{\frac{2}{3}}$$
	and $$
	g(n)!\leq F(n)^{\frac{3}{2}}.
	$$
	Furthermore, as $F(n)$ tends to infinity, we have that for sufficiently large $n$ that$$
	\frac{g(n)!}{2}\geq F(n)^{\frac{1}{2}}.
	$$
	After repeatedly applying \ref{cond:localgrowth1}, we obtain that there exist constants $C_7,C_8$ such that for $n$ sufficiently large we have $$
	\frac{g(C_8n)!}{2}\geq F(n)
	$$
	and $$
	\frac{g(n)!}{2}\leq F(C_7n)
	$$
	for $n$ sufficiently large.
	Enlarging these constants if necessary, we obtain that this also holds for all $n$.
\end{proof}
For the function $d$ as above, the groups we construct will be a $2$-generated subgroup of 
$$
G_0(d)=\left(\prod_{m=1}^\infty \mathrm{Alt}(d(m))\right)\times \ZZ_3\wr\ZZ.
$$
On this group we denote the natural projections$$
\pi_m:G_0(d)\rightarrow \alt(d(m))
$$
and$$
\pi_\infty:G_0(d)\rightarrow\ZZ_3\wr\ZZ.
$$

To define our generators, we use a second function.
By \cite[Proposition 2.8.]{Bradford} we can find some $r:\NN\rightarrow\NN$ such that:
\begin{enumerate}[label=(\arabic*)]
	\item For all $n\in\NN,n<r(n)<18n$ and $r(n)<d(n)/3;$\\
	\item For all $l,m\in\NN$, if $l\neq m$ then $r(l)\not\equiv\pm r(m),\pm2r(m)\mod d(m).$
\end{enumerate}

We now construct two generators $\alpha_*$ and $\beta_*$ as follows:
for natural numbers $d$ and $r$, let $\alpha_d,\beta_{r,d}\in\mathrm{Alt}(d)$ be given respectively, by the $d$ cycle $$
\alpha_d=(0,1,\cdots, d-1)
$$
and the $3$ cycle
$$
\beta_{r,d}=(0,r,2r).
$$
Furthermore let $\alpha_\infty,\beta_\infty\in\ZZ_3\wr\ZZ$ be given by:$$
\alpha_\infty=(0,1)
$$
and $$
\beta_\infty=1\in\ZZ_3<\ZZ_3\wr\ZZ
$$

	Given $\alpha_{d}$ and $\beta_{r,d}$ defined as before,
then $\alpha_*=((\alpha_{d(m)})_m,\alpha_\infty)$ and $\beta_*=((\beta_{r(m),d(m)})_m,\beta_\infty)$ are elements of $$
	G_0(d).
	$$
	Define then $G(d,r)$ the subgroup of $G_0(d)$ generated by $S=\{\alpha_*,\beta_*\}$.
	
	The notation $\alpha_\infty$ en $\beta_\infty$ is suggestive of a correspondence between these elements and $\alpha_d$ and $\beta_{r,d}$. This is made more concrete in the following proposition:	
	\begin{proposition}{\cite[Proposition 2.12.]{Bradford}}\label{prop:WordproblemIsLocal}
		Let $w\in F_{a,b}$ be a freely reduced word of length at most $n$ in the variables $x$ and $y$. Let $r,d\in\NN$ and suppose
		$$
		d-2r\geq 2n+1
		$$
		and$$
		r\geq 2n+1.
		$$
		Then $w(\alpha_d,\betha_{r,d})=e$ in $\alt(d)$, if and only if $w(\alpha_\infty,\beta_\infty)=e$ in $\ZZ_3\wr\ZZ$.
	\end{proposition}

	Notice that this construction differs slightly from the groups described in \cite[Section 3.1]{Bradford}. We show however that these groups are in fact the same.
	\begin{lemma}
		The canonical projection $$\pi:\prod_{i=1}^\infty \mathrm{Alt}(d(m))\times \ZZ_3\wr\ZZ\rightarrow \prod_{m=1}^\infty \mathrm{Alt}(d(m))
		$$
		induces an isomorphism from $G(d,r)$ to $B(d,r,r)=\langle(\alpha_{d(m)})_m,(\beta_{r(m),d(m)})_m\rangle$.\footnote{These groups $B(d,r,r)$ are the groups constructed in \cite{Bradford}.}
	\end{lemma}
	\begin{proof}
		It is immediate that $\pi$ is surjective as $\pi$ maps the generators $((\alpha_{d(m)})_m,\alpha_\infty)$ and $((\beta_{r(m),d(m)})_m,\beta_\infty)$ to $(\alpha_{d(m)})_m$ and $(\beta_{r(m),d(m)})_m$, the generators of $B(d,r,r)$, respectively. It thus suffices to show that this map has trivial kernel. For this, let $w\in F_{a,b}$ such that $$
			\pi(w(\alpha_*,\beta_*))=e,
			$$
			that is, such that $w(\alpha_*,\beta_*)$ is an element of $G(d,r)$ that lies in the kernel of $\pi$.
			Let $n$ be the word norm of $w$ in $F_2$ and let $m\geq2n+1$. We have that $r(m),d(m)-2r(m)\geq m\geq 2n+1$, thus by \cref{prop:WordproblemIsLocal}, $w(\alpha_\infty,\beta_\infty)$, vanishes precisely when $w(\alpha_{d(m)},\beta_{d(m)})$ vanishes. As $\pi(w(\alpha_*,\beta_*))=1$, this is indeed the case, and thus we have that $w(\alpha_\infty,\beta_\infty)$ vanishes. We thus have shown that the kernel of $\pi:G(d,r)\rightarrow B(d,r,r)$ is trivial and thus that $\pi$ induces an isomorphism.
	\end{proof}
	We can thus without problems use results from $B(d,r,r)$ on $G(d,r)$.
	In particular,
	we have a form of the main conclusion of \cite{Bradford};
	\begin{theorem}{\cite[Theorem 3.10.]{Bradford}}\label{prop:bradMainResult}
		Let $F,d$ and $r$ be as before, then $$\mathrm{Rf}_{G(d,r)}^S(n)\simeq F(n).$$
	\end{theorem}
	The two remaining lemmas that are most important for our purposes are:
	
	\begin{lemma}{\cite[Lemma 3.1.]{Bradford}}\label{prop:ababVanishes}
		For $m,n\in\NN$, $\beta_{r(n),d(n)}$ and $\alpha_{d(n)}^{r(m)}\beta_{r(n),d(n)}\alpha_{d(n)}^{-r(m)}$ commute if and only if $m\neq n$.
	\end{lemma}
	
	\begin{proposition}{\cite[Corollary3.3.]{Bradford}}\label{prop:ContainsAlt}
		For all $m\in \NN$,  and for $$T_m=\alt(d(m))\leq
		\prod_{m=1}^\infty \mathrm{Alt}(d(m))\times \ZZ_3\wr\ZZ,
		$$ we have $T_m\leq G(d,r)$.
	\end{proposition}
	\begin{corollary}\label{prop:TorsionByCyclic}
		The group $G(d,r)$ above is an extension of the torsion group $\bigoplus_{n\in\NN}\alt(d(n))$ by the group $\ZZ_3\wr\ZZ$.
	\end{corollary}
	\begin{proof}
		By \cref{prop:WordproblemIsLocal}, we have that the kernel of the canonical projection $G(d,r)\rightarrow\ZZ_3\wr\ZZ$ is precisely $G(d,r)\cap\bigoplus_{n\in\NN}\alt(d(n))$. By \cref{prop:ContainsAlt} we have that $\bigoplus_{n\in\NN}\alt(d(n))\subset G(d,r)$ and thus that the kernel is equal to $\bigoplus_{n\in\NN}\alt(d(n))$.
	\end{proof}
	\section{Bound of the conjugacy separability}\label{sec:conjOnBrad}
	For the residual finiteness growth, the proof only considers the quotients of the form $\alt(d(n))$.
	For the conjugacy separability growth however, it will not suffice to consider the projections on the groups $\mathrm{Alt}(d(m))$. We demonstrate however that it is possible to separate conjugacy classes if we also allow projection on $\ZZ_3\wr\ZZ$.
	\begin{proposition}\label{prop:ConjIsLocal}
		There exists some constant $K$ such that whenever $g_1,g_2\in G(d,r)$ are such that $\norm{g_1}_S,\norm{g_2}_S\leq n$, then $g_1$ and $g_2$ are conjugate if and only if both:\begin{itemize}
			\item $\pi_\infty(g_1)$ and $\pi_\infty(g_2)$ are conjugate in $\ZZ_3\wr\ZZ$.\\
			\item For all $m\leq Kn$, $\pi_m(g_1)$ and $\pi_m(g_2)$ are conjugate in $\mathrm{Alt}(d(m))$.
		\end{itemize}
	\end{proposition}
	\begin{proof}
		That the former implies the latter is clear.
		For the other direction let $w_1,w_2\in F_{a,b}$ such that $w_1(\alpha_*,\beta_*)=g_1$ and $w_2(\alpha_*,\beta_*)=g_2$. Both $w_1$ and $w_2$ can be chosen of length at most $n$.
		By the first assumption, $\pi_\infty(g_1)$ and $\pi_\infty (g_2)$ are conjugate in $\ZZ_3\wr\ZZ$. By \cite[Theorem 1.]{Salle}, there exists some $g_0\in \ZZ_3\wr \ZZ$ such that $g_0\pi_\infty(g_1)g_0\inv=\pi_\infty (g_2)$ and such that $\norm{g_0}_{S_\infty}\leq K_0n$ for some uniform constant $K_0$. We thus find some word $w_0\in F_2$ of length at most $K_0n$ such that $w_0(\alpha_\infty,\beta_\infty)=g_0$. Notice that we have $$
		(w_0*w_1*w_0\inv*w_2\inv)(\alpha_\infty,\beta_\infty)=e.
		$$
		Let $m\geq 4K_0n+4n+1$, as the length of $(w_0*w_1*w_0\inv*w_2\inv)$ is at most $2K_0n+2n$, we have by \cref{prop:WordproblemIsLocal} that the above holds if and only if $$
		(w_0*w_1*w_0\inv*w_2\inv)(\alpha_{d(m)},\beta_{r(m),d(m)})=e
		$$
		and thus we have $$
		w_0(\alpha_{d(m)},\beta_{r(m),d(m)})\pi_m(g_1)w_0(\alpha_{d(m)},\beta_{r(m),d(m)})\inv=\pi_m(g_2).
		$$
		
		If we let $K=4K_0+4$, then by the second assumption, for all $m\leq 4K_0n+4n$ we have that $\pi_m(g_1)$ and $\pi_m(g_2)$ are conjugate. By transitivity, we also have that $\pi_m(g_2)$ is conjugate to$$
		w_0(\alpha_{d(m)},\beta_{r(m),d(m)})\pi_m(g_1)w_0(\alpha_{d(m)},\beta_{r(m),d(m)})\inv
		$$
		Let $h_m$ be the conjugator, such that $$
		h_mw_0(\alpha_{d(m)},\beta_{r(m),d(m)})\pi_m(g_1)w_0(\alpha_{d(m)},\beta_{r(m),d(m)})\inv h_m\inv=\pi_m(g_2).
		$$
		By \cref{prop:ContainsAlt} we have that $\mathrm{Alt}(d(m))$ is a subgroup of $G(d,r)$ in the natural way.
		Let now $h=\prod_{m=1}^{4K_0n+4n} h_m$, then we finally have that $$
		\left(hw_0(a_*,b_*)\right)g_1\left(hw_0(a_*,b_*)\right)\inv=g_2
		$$
		or thus that $g_1$ and $g_2$ are conjugate as had to be proven.		
	\end{proof}
	\setcounter{result}{0}
	\begin{result}\label{prop:exactupperbound}
		Let $F:\NN\rightarrow\NN$ be a non-decreasing function satisfying \ref{cond:globalgrowth1} and \ref{cond:localgrowth1}.
		Then there exists a conjugacy separable group $G$, generated by a finite set $S$, such that $\conj_G^S(n)\simeq F(n)$
	\end{result}
	\begin{proof}
		Let $G$ be the group $G(d,r)$ where $d,r$ are as \cref{sec:bradconstruct}.
		The lower bound follows from \cref{prop:bradMainResult}, combined with the fact that $$ 
		\rf_G^S(n)\leq \conj_G^S(n).
		$$
		For the upper bound,
		let $g_1, g_2\in G$ be non-conjugate elements such that $\norm{g_1}_S,\norm{g_2}_S\leq n$ for some integer $n$. Then by \cref{prop:ConjIsLocal} either there exists some $m\leq Kn$ such that $\pi_m(g_1)$ and $\pi_m(g_2)$ are non-conjugate, or $\pi_{\infty}(g_1)$ and $\pi_{\infty}(g_2)$ are non-conjugate.
		
		In the first case, we have found a finite quotient of size at most $\abs{\mathrm{Alt}(d(m))}\leq d(Kn)!/2$ in which the images of $g_1$ and $g_2$ are non-conjugate. We thus have $$
		\conj_G(g_1,g_2)\leq (2f(Kn))!/2.
		$$

		On the other hand, assume that $\pi_\infty(g_1)$ and $\pi_\infty(g_2)$ are non-conjugate. By \cite{Ferov}, we then also have some constant $c_4$, not depending on $g_1$ and $g_2$ such that there exists some finite quotient $\rho:\ZZ_3\wr\ZZ\rightarrow Q$ such that $Q$ is finite of order no more then $\exp(c_4n)$ ,and such that $\rho\comp\pi_\infty(g_1)$ and $\rho\comp\pi_\infty(g_2)$ are non-conjugate. In the second case we thus have that $\conj(g_1,g_2)\leq \exp(c_4n)$.
		For $n$ sufficiently large we have that $\exp(c_4n)\leq (2f(Kn))!$ and thus we have that for $n$ sufficiently large that $\conj_G^S(n)\leq (2f(Kn))!$. By \cref{prop:inverseFactorialAproximation} it follows that $\conj_G(n)\prec F(n)$.
	\end{proof}

	By \cref{prop:TorsionByCyclic}, we have that the groups considered above are all torsion-by-cyclic. This torsion helps a lot in restricting precisely which quotients are possible and which are not. Also in \cite{Bou-Rabee} is the residual finiteness growth controlled by finding elements in larger and larger alternating subgroups. The groups of arbitrarily complex word-problem and arbitrarily large residual finiteness growth constructed in \cite{Kharlampovich} are also extensions of torsion groups by abelian groups.
	The torsion subgroups in these groups seem essential to control their residual finiteness growth.
	\begin{question}
		For which functions $f:\NN\rightarrow\NN$ do there exist torsion-free residually finite$/$ conjugacy separable  groups $G$ such that $\mathrm{Rf}_G(n)/\conj_G(n)\simeq f(n)$?
	\end{question}

	\section{The gap between $\rf$ and $\conj$}
		In the previous class of examples, the residual finiteness growth and the conjugacy separability growth where always equal. In general groups however, this need not be the case. Indeed in this section we give a class of groups with arbitrary residual finiteness growth and arbitrary conjugacy separability growth, where these two are chosen independently. To achieve this, we construct a new class of groups, indexed by a function $d:\NN\times\NN\rightarrow \NN$ and a function $r:\NN\rightarrow\NN$.
		
		\subsection{The new construction}\label{sec:constructingGroups}
		Let $d:\NN^2\rightarrow\NN$ be an odd prime valued function and suppose there exists some function $m_*:\NN\rightarrow\NN:n\mapsto m_n$ such that for $m\geq m_n, d(n,m)=d(n,m_n)$. \footnote[2]{In other words, $d(n,\_)$ stabilizes after the $m_n$'th entry.}
		We then define the group $G_0(d,m_*)$ as the product
		$$
		G_0(d,m_*)=\left(\prod_{n=1}^\infty \prod_{m=1}^{m_n} \mathrm{Alt}(d(n,m))\right)\times \ZZ_3\wr\ZZ.
		$$
		\newcommand{\Gn}{G_0(d,m_*)}
		The projection maps are denoted $\pi_{n,m}:\Gn\rightarrow\alt(d(n,m))$ and $\pi_\infty:\Gn\rightarrow \ZZ_3\wr\ZZ$.
		Notice that $d$ needs not be an injective function and as such is it possible that $\Gn$ has multiple direct factors isomorphic to $\alt(d(n,m))$. In the sequel when we talk about \emph{the} subgroup $\alt(d(n,m))$, then we are referring to the direct factor $\alt(d(n,m))$ which appears in the product above with indices $(n,m)$.
		
		To construct a finitely generated subgroup of $\Gn$, we make use of a second function $r:\NN\rightarrow\NN$. This function $r$ should be such that $2r(n)+1\leq d(n,m)$ for any choice of $n,m\in\NN$.
		
		In the group $G_0(d,m_*)$ we define the two elements $\alpha_*$ and $\beta_*$ as
		$\alpha_*=(((\alpha_{d(n,m)})_m)_n,\alpha_\infty)$ and $\beta_*=(((\beta_{r(n)d(n,m)})_m)_n,\beta_\infty)$.
		We define the group $G(d,m_*,r)$ as the subgroup 
		$$G(d,m_*,r)=\mathrm{span}\{\alpha_*,\beta_*\}$$
		of $G_0(d,m_n)$ and fix its generating set $S=\{\alpha_*,\beta_*\}$.
		\newcommand{\G}{G(d,m_*,r)}
		Notice that the projection maps are such that $\pi_{n,m}(\alpha_*)=\alpha_{d(n,m)}$ and $\pi_{n,m}(\beta_*)=\beta_{r(n)}$.
		
		The group $\G$ depends greatly on the functions $d$ and $r$, for instance $\G$ is finite if and only if $d$ is bounded.
		For our purposes we place some restrictions on $r$ and $d$.
	
		First on $d$ we place the following conditions:
		\begin{enumerate}[label=(\alph*)]
			\item $d(\_,1)$ is a non-decreasing function;\label{cond:nondecreasing}
			\item For all $n,m\in\NN$, $d(n,m)$ is an odd prime number;\label{cond:primes}\\
			\item For all $n,m\in\NN$ with $m<m_n$, and for some constants $C_1>C_0>1$, $C_1d(n,m)>d(n,m+1)>C_0d(n,m)$;\label{cond:expGrowth}\\
			\item For all $n\in\NN$ and for some constant $C_2>0$, $d(n,1)\geq \max(C_2n\log(n)(\log\log(n))^{1+\epsilon},C_2)$.\label{cond:dsuflarge}\\
		\end{enumerate}
	
	On $r$ on the other hand, we impose the conditions:
	\begin{enumerate}[label=(\arabic*')]
		
		\item For all $n\in\NN, n\leq r(n)< 37n$ and $r(n)< \frac{d(n,1)}{6}$;\label{cond:RAproxN}
		\item For all $l,m,n\in\NN$, if $l\neq n$ then $r(l)\not\cong\pm r(n),\pm2r(n)\mod d(n,m);$\label{cond:RNonCongruent}
		\item For all $n\in \NN$, $r(n)\cong 1\mod 6$.\label{cond:RCong1}
	\end{enumerate}
	
	A priori, it might seem unclear if such functions $r$ exist. The restrictions on $d$ however guarantee that this is the case:
	\begin{proposition}\label{prop:RExistenceDoubleIndexed}
		For all $\epsilon>0,C_0>1$ there exists some $C_2>0$ such that the following holds. For any $d:\NN\times\NN\rightarrow \NN$  satisfying \ref{cond:nondecreasing} and \ref{cond:primes}, such that \ref{cond:expGrowth} holds for some $C_1>C_0$ and such that \ref{cond:dsuflarge} holds for $\epsilon$ and $C_2$,
		 there exists a function $r:\NN\rightarrow\NN$ satisfying \ref{cond:RAproxN},\ref{cond:RNonCongruent} and \ref{cond:RCong1}
	\end{proposition}
	\begin{proof}
		First notice that if $d(n,m)$ and $d(n,m')$ are equal, then $r(l)\not\cong\pm r(n),\pm2r(n)\mod d(n,m)$ and $r(l)\not\cong\pm r(n),\pm2r(n)\mod d(n,m')$ are equivalent and thus does $d(n,m')$ not add an extra constraint. For \ref{cond:expGrowth}, we may thus assume that $d(n,m+1)$ is always larger then $C_0d(n,m)$.
		
		Let $d_0:\NN_{\geq 3}\rightarrow\RR_{>0}$ be the function mapping $n$ to ${n\log(n)(\log\log(n))^{1+\epsilon}}$. 
		As $d_0(n)$ is increasing, we have by the Cauchy condensation test that \begin{align*}
			\sum\frac{1}{n\log(n)(\log\log(n))^{1+\epsilon}}
			&\leq \sum \frac{2^n}{2^n\log(2^n)(\log\log(2^n))^{1+\epsilon}}\\
			&\leq  \sum\frac{2^{n}}{2^n\log(2)(n\log(2)+\log\log(2))^{1+\epsilon}}.
		\end{align*}
		The last of these converges, so $\sum\frac{1}{d_0(n)}$ must converge as well.
		Given $M>1$ some integer (to be chosen), taking $C_2$ sufficiently large guarantees that $$
		\sum_{n=1}^\infty\frac{1}{d(n,1)}\leq \frac{1-1/C_0}{36M}.
		$$
		Increasing $C_2$ if necessary, we may assume that $d(n,1)\geq 4(M+1)n$.
		We define $r$ inductively. Let $r(1),\cdots, r(n-1)$ be integers such that \ref{cond:RAproxN},\ref{cond:RNonCongruent} and \ref{cond:RCong1} hold for $l,i\leq n-1$. We show that for a good choice of $M$, we can find $r=r(n)$ such that these conditions still hold.
		
		First we count the number of integers $r$ such that $n\leq r<(M+1)n$ and such that there exists some $i<n$ and some $m$ such that $r(i)\cong \pm r,\pm2r\mod d(n,m)$. As $d(n,m)\geq 4(M+1)n$, we have that this is the case precisely when $r(i)=r$ or $ r(i)=2r$ for some $i$, there are thus at most $2n$ such values.
		
		Now we count the number of integers $r$ such that $n\leq r<(M+1)n$ and such that there exists some $i<n$ and some $m$ satisfying$$
		r\cong \pm r(i),\pm2r(i)\mod d(i,m).
		$$
		We distinguish $2$ cases: either $d(i,m)>4(M+1)n$ or $d(i,m)\leq 4(M+1)n<8Mn$.
		In the first case we again have that the equality holds precisely when $ r=r(i)$ or $r=2r(i)$. There are thus at most $2n$ such values in this case.
		
		For the second case, for a fixed value of $i$ and $m$, there are at most $4\lceil\frac{Mn}{d(i,m)}\rceil\leq4\frac{Mn}{d(i,m)}+4$ such values of $r$. By the assumption, this is bounded above by $4\frac{Mn}{d(i,m)}+4\frac{d(i,m)}{d(i,m)}\leq36\frac{Mn}{d(i,m)}$. We now sum over all possible values $m$, to obtain the number of values $r$ such that for a fixed $i$, there exists some $m$ such that $r\cong \pm r(i),\pm2r(i)\mod d(i,m)$. In the second argument, $d(\_,\_)$ grows at least as fast as a geometric sequence, we obtain that there are at most $$
		\frac{36Mn}{d(i,1)(1-1/C_0)}
		$$
		such values $r$. Now we may sum over all values $i$ and obtain that there are at most $n$
		values of $r$ such that $r\cong \pm r(i),\pm2r(i)\mod d(i,m)$ for some $i$ and $m$.
		
		Notice that there are at least $\lfloor \frac{Mn}{6}\rfloor$ integers in the interval $[n,(M+1)n)$ that are congruent to $1\mod 6$. Using all the previous, we thus have that there are at least$$
		\left\lfloor \frac{Mn}{6}\right\rfloor-2n-2n-n
		$$
		integers $r\cong 1\mod 6$ such that $r\not\cong\pm r(i),\pm2r\mod d(i,m)$ and such that $r(i)\not\cong\pm r,\pm2r(i)\mod d(n,m)$.
		
		In particular, for $M=36$, there exists at least one such value for $r$.

	\end{proof}
	
	A first consequence of the above constraints is that the alternating groups $\alt(d(n,m))$ form subgroups of $\G$.
	
	\begin{lemma}\label{prop:AltAreSubgroups}
		Let $d,r$ be such that \ref{cond:primes}, \ref{cond:expGrowth} and \ref{cond:RNonCongruent} hold.
		The natural inclusion of $\alt(d(n,m))$ into $\Gn$ is a subset of $\G$.
	\end{lemma}
	\begin{proof}
		Take $n$ arbitrary, we proceed by induction on $m$.
		We first show that ${\alt(d(n,m))\cap \G}$ is non-trivial.

		Denote $$v_i=[a^{r(i)}b\inv a^{-r(i)},b\inv]\in F_{a,b}.$$
		By \cref{prop:ababVanishes}, we have that $v_i(\alpha_{d(n,m)},\beta_{r(n),d(n,m)})$ vanishes when $i\neq n$ and in $\ZZ_3\wr\ZZ$ we have that $v_i(\alpha_\infty,\beta_\infty)$ vanishes. Conversely, when $n=i$, then $v_i(\alpha_{d(n,m)},\beta_{r(n),d(n,m)})$ does not vanish.
		Consider now $$w_{i,j}=[a^{d(i,j)-3r(i)}v_i(a,b)a^{-d(i,j)+3r(i)},v_i(a,b)]\in F_{a,b}.$$
		As $v_i(\alpha_{d(n,m)},\beta_{r(n),d(n,m)})$ vanishes for $n\neq i$, then certainly $w_{i,j}(\alpha_{d(n,m)},\beta_{r(n),d(n,m)})$ must vanish as well.
		Similarly, $w_{i,j}(\alpha_\infty,\beta_\infty)$ vanishes. So we are left with the case where $i=n$.
		First suppose that $(i,j)=(n,m)$. A computation shows that in this case $w_{n,j}(\alpha_{d(n,m)},\beta_{r(n),d(n,m)})$ is precisely the permutation in $\alt(d(n,m))$ with cycle notation$(0,d(n,m)-3r(n),3r(n))$.
		On the other hand if $m>j$, then the left part of the commutator permutes the points $d(n,j)-3r(n), d(n,j)-2r(n), d(n,j)-r(n)$ and $d(n,j)$ all of which lie in the interval $[3r(n)+1,d(n,m)-1]$ and the right permutes the points $0,r(n),2r(n)$ and $3r(n)$, non of which lie in that interval. It follows that in this case that the commutator vanishes.
		
		From all of the above, it follows that $w_{n,m}(\alpha_*,\beta_*)\in \G$ is an element of $$
		\prod_{i=1}^m\alt(d(n,i))
		$$
		and thus equal to$$
		\prod_{i=1}^{m} w_{n,m}(\alpha_{d(n,i)},\beta_{r(n),d(n,i)}).
		$$
		By the induction hypothesis, we have that  $$
		\prod_{i=1}^{m-1} w_{n,m}(\alpha_{d(n,i)},\beta_{r(n),d(n,i)})\in\prod_{i=1}^{m-1}\alt(d(n,i))
		$$
		is also an element of $\G$. We thus have that $w_{n,m}(\alpha_{d(n,m)},\beta_{r(n),d(n,m)})\in \alt(d(n,m))$ is a non-trivial element of $\G$.
		
		We now show that $\alt(d(n,m))\cap \G$ is a non trivial normal subgroup of $\alt(d(n,m))$. Non-triviality follows from the previous, and the fact that it must be a subgroup is obvious, it thus remains to be shown that this intersection is closed under conjugation by elements of $\alt(d(n,m))$. Let $g\in \alt(d(n,m))\cap\G$ and let $h\in \alt(d(n,m))$ be arbitrary. By \cite[1.3.]{Lewin1975generating}, $\alt(d(n,m))$ is generated by $\alpha_{d(n,m)}$ and $\beta_{r(n)}$. In particular, there exists some word $w\in F_{a,b}$ such that $w(\alpha_{d(n,m)},\beta_{r(n),d(n,m)})=h$. We demonstrate that $w(\alpha_*,\beta_*)gw(\alpha_*,\beta_*)\inv=hgh\inv$.
		As $\alt(d(n,m))$ is a normal subgroup of $\Gn$, we have that $w(\alpha_*,\beta_*)gw(\alpha_*,\beta_*)\inv\in\alt(d(n,m))$. As furthermore $\alt(d(n,m))$ is a direct summand of $\Gn$, we have that this element remains the same after first projecting onto $\alt(d(n,m))$ and then embedding back into $\Gn$. We thus have that \begin{align*}
			w(\alpha_*,\beta_*)gw(\alpha_*,\beta_*)\inv&=\pi_{n,m}(w(\alpha_*,\beta_*)gw(\alpha_*,\beta_*)\inv)\\
			&=w(\alpha_{n,m},\beta_{n,m})gw(\alpha_{n,m},\beta_{n,m})\inv.
		\end{align*}
		The latter is then equal to $hgh\inv$. We have thus shown that $hgh\inv\in \G$ and as this holds for any $g,h$, we have that $\alt(d(n,m))\cap\G$ is normal in $\alt(d(n,m))$. Finally, as $\alt(d(n,m))$ is a simple group, we have that $\alt(d(n,m))\cap\G=\alt(d(n,m))$ and thus that $\alt(d(n,m))\subset \G$.
		
		Via induction, the result follows for all $m$.

	\end{proof}

		\begin{lemma}\label{prop:ConjIsLocal2}
		There exists some constant $K$ such that whenever $g_1,g_2\in \G$ are such that $\norm{g_1}_S,\norm{g_2}_S\leq N$, then $g_1$ and $g_2$ are conjugate, if and only if both:
		\begin{itemize}
			\item $\pi_\infty(g_1)$ and $\pi_\infty(g_2)$ are conjugate in $\ZZ_3\wr\ZZ$;
			\item For all $n\leq KN$ and for all $m, \pi_{n,m}(g_1)$ and $\pi_{n,m}(g_2)$ are conjugate in $\alt(d(n,m))$.
		\end{itemize}
	\end{lemma}
	\begin{proof}
		The proof is identical to the proof of \cref{prop:ConjIsLocal} after replacing \cref{prop:ContainsAlt} by \cref{prop:AltAreSubgroups}.
	\end{proof}
	This above result allows us to study elements of $\G$ by considering them in a finite number of quotients. Another result that restricts the list of quotients even further is the following.
		\begin{lemma}\label{prop:AltIsDInvariant}
		Let $w\in F_{a,b}$ and let $n,m_1,m_2$ be such that $d(n,m_1)-2r(n),d(n,m_2)-2r(n)\geq 2\norm{w}+2$. We have that
		$$
		w(\alpha_{d(n,m_1)},\beta_{r(n)})\in\alt(d(n,m_1))
		$$
		is trivial if and only if $$
		w(\alpha_{d(n,m_2)},\beta_{r(n)})\in\alt(d(n,m_2))
		$$
		is trivial.
	\end{lemma}
	\begin{proof}
		For ease of notation, we identify the points $\{0,1,\cdots,d(n,m_j)-1\}$ with $\frac{\ZZ}{d(n,m_j)\ZZ}$. Under this identification we have that $\alpha_{d(n,m_j)}x=x+1$.
		
		We may rewrite $w$ in the form$$
		a^{q_0} (a^{q_1}b^{\epsilon_1} a^{-q_1})(a^{q_2}b^{\epsilon_2} a^{-q_2})\cdots(a^{q_l}b^{\epsilon_l} a^{-q_l})
		$$
		The exponents $q_i$ are all bounded in absolute value by $\norm{w}$.
		Notice that the permutation, $(a^{q_i}b a^{-q_i})(\alpha_{d(n,m_j)},\beta_{r(n),d(n,m_j)})$ has reduced cycle notation $(q_i,q_i+r(n),q_i+2r(n))$.
		In particular, we have that $$
		(a^{q_1}b^{\epsilon_1} a^{-q_1})(a^{q_2}b^{\epsilon_2} a^{-q_2})\cdots(a^{q_l}b^{\epsilon_2} a^{-q_l})(\alpha_{d(n,m_j)},\beta_{r(n),d(n,m_j)})
		$$ acts trivially outside of the interval. $[-\norm{w},\norm{w}+2r(n)].$
		Thus if $d(n,m_j)\geq 2\norm{w}+2r(n)+2$, then there exists some point $x$ (for instance $\norm{w}+2r(n)+1$) outside this interval. It follows that $w(\alpha_{d(n,m_j)},\beta_{r(n),d(n,m_j)})$ acts on $x$ as $\alpha_{d(n,m_j)}^{q_0}$.
		As $w(\alpha_{d(n,m_1)},\beta_{r(n),d(n,m_j)})$ acts trivial, it must map $x$ to $x$ or thus must $q_0$ be a multiple of $d(n,m_1)$. As $\abs{q_0}< d(n,m_1)$ we have that $q_0=0$ or thus that $w$ is of the form $$
		(a^{q_1}b^{\epsilon_1} a^{-q_1})(a^{q_2}b^{\epsilon_2} a^{-q_2})\cdots(a^{q_l}b^{\epsilon_2} a^{-q_l}).
		$$
		In $\frac{\ZZ}{d(n,m_1)\ZZ}$ and $\frac{\ZZ}{d(n,m_2)\ZZ}$, we identify the intervals $[-\norm{w},\norm{w}+2r(n)]$ with one another in the natural way. Notice that as $q_i$ is bounded by $\norm{w}$, we have that this identification preserves the action of words $a^{q_i}b^{\epsilon_i} a^{-q_i}$. That is if $\varphi$ is the identification and $x$ is one of the points in this interval, then \begin{align*}&\varphi\comp(a^{q_i}b^{\epsilon_i} a^{-q_i})(\alpha_{d(n,m_1)},\beta_{r(n),d(n,m_1)})(x)\\=&(a^{q_i}b^{\epsilon_i} a^{-q_i})(\alpha_{d(n,m_2)},\beta_{r(n),d(n,m_2)})\comp\varphi(x).\end{align*}
		If now $w(\alpha_{d(n,m_1)},\beta_{r(n),d(n,m_1)})$ acts trivially, then it must also act trivially on ${[-\norm{w},\norm{w}+2r(n)]}$ but by the previous, this implies that $w(\alpha_{d(n,m_2)},\beta_{r(n),d(n,m_2)}$ acts trivially on $[-\norm{w},\norm{w}+2r(n)]$, the result follows as $q_0=0$ and thus that $w(\alpha_{d(n,m_2)},\beta_{r(n),d(n,m_2)})$ always acts trivially outside of $[-\norm{w},\norm{w}+2r(n)]$.
		
	\end{proof}
	
	Note that the above statement does not say anything about conjugacy. Indeed conjugacy of elements is more sensitive to a change of $r(n)$ and $d(n,m)$. A very concrete example of this is the following.
	\begin{lemma}\label{prop:g1g2Conjugate}
		Let $g_1,g_2\in F_{a,b}$ where$$
		g_1=a^3b\inv[a^rb\inv a^{-r},b\inv]b\inv
		$$
		and $$
		g_2=a^3b[a^rb\inv a^{-r},b\inv]
		$$
		for some integers $r>3$, $d>4r$.\begin{itemize}
			\item If $r\cong 1\mod 6$, $d\cong1\mod 6$ then $g_1(\alpha_d,\beta_{r,d})$ and $g_2(\alpha_d,\beta_{r,d})$ are conjugate.\item If $r\cong 1\mod 6$, $d\cong5\mod 6$ then $g_1(\alpha_d,\beta_{r,d})$ and $g_2(\alpha_d,\beta_{r,d})$ are non-conjugate.\end{itemize}
		
	\end{lemma}
	\begin{proof}
		This is straightforward but somewhat tedious computation. The proof is presented in \cref{sec:twoElements}.
	\end{proof}

			\subsection{Computing the separability functions}
			
			In this section we prove the following theorem:
			\begin{result}\label{prop:conjGap}
			Let $F_1$ and $F_2$ be two functions satisfying \ref{cond:globalgrowth1} and \ref{cond:localgrowth1}
			then there exists some group $G$ such that $\mathrm{Rf}_G(n)\simeq F_1(n)$ and $\conj_G(n)\simeq F_2(n)$.
			\end{result}
			For this we will use the groups constructed above, the behaviour of $\rf_G$ will be dictated by $d(n,1)$ and the behaviour of $\conj$ by $d(n,m_n)$.

			Let $F_1$ and $F_2$ be as above. Let $C_2$ be the constant from \cref{prop:RExistenceDoubleIndexed} for $\frac{\epsilon}{2}$ and $C_0=2$. By \cite[Lemma 2.3.]{Bradford}, there exists some constant $C_3$ such that 
			$$
			f_1(n)=\left\lceil\frac{\log F_1(n+C_3)}{\log\log F_1(n+C_3)}\right\rceil
			$$
			satisfies $$
			f_1(n)\geq \max(C_2n\log(n)\log\log(n)^{1+\frac{\epsilon}{2}},C_2)
			$$
			Similarly define $f_2=\left\lceil\frac{\log F_2(n+C_3)}{\log\log F_2(n+C_3)}\right\rceil$.
			We have $f_2(n)\geq f_1(n)$.
			We want to approximate $f_2(n)$ and $f_1(n)$ by sequences of prime numbers. For this we use a generalisation of Bertrand's postulate
			\begin{proposition}[Bertrand's postulate on arithmetic progressions]\label{prop:bertrand}
			There exists some constant $C_6$ such that for any $n\in \NN$ and for $i=\pm 1$, there exists some prime number $p$ such that $p\cong i\mod 6$ and $n<p<C_6n$.\footnote[2]{More in general, we may replace $i$ and $6$ by any pair of coprime integers.}
			\end{proposition}
			\begin{proof}
				See for instance \cite{moree1993bertrands}.
			\end{proof}
			
			Using the above, we can now find functions $d_1,d_2:\NN\rightarrow\NN_{>5}$ both consisting of odd prime numbers such that:\begin{itemize}
				\item $f_1(n)\leq d_1(n)\leq C_6f_1(n)$ and $d_1\cong 1\mod 6$;\\
				\item $f_2(n),2d_1(n)\leq d_2(n)\leq 2C_6^2f_2(n)$ and $d_2\cong 5\mod 6$,\\
			\end{itemize}
			where $C_6$ is as in \cref{prop:bertrand}
			We now construct functions $d:\NN\times\NN\rightarrow \NN$ and $m_*:\NN\rightarrow\NN$ satisfying the conditions of \cref{prop:RExistenceDoubleIndexed} for $\frac{\epsilon}{2}$, $C_0=2, C_1=C_6$, and $C_2$ as before, such that 
			\begin{enumerate}
				\item $d(n,1)=d_1(n)$;\\
				\item $d(n,m_n)=d_2(n)$;\\
				\item $d(n,m)\cong 1\mod 6$ whenever $m<m_n$.\\
			\end{enumerate}
			
			We construct $d(n,m)$ by induction on $m$.
			For the base case we have $d(n,1)=d_1(n)$ with $d_2(n)\geq 2d(n,1)$. Now given $d(n,m-1)$ where $2d(n,m-1)\leq d_2(n)$ we define $d(n,m)$ as follows: If $d_2(n)\leq 4C_6d(n,m-1)$ then $d(n,m)=d_2(n)$, in this case we also let $m=m_n$ and $d(n,i)=d_2(n)$ for all $i\geq m$. On the other hand, if $4C_6d(n,m-1)<d_2(n)$, take then $d(n,m)$ the least prime $p$ such that $p\cong 1\mod 6$ and such that $p\geq 2d(n,m-1)$. By the generalised Bertrand's postulate there exists such a prime in the interval $[2d(n,m-1),2C_6d(n,m-1)]$, in this case, we still have that $2d(n,m)\leq d_2(n)$.
			
			Notice that if we only encounter the second case, $d(n,\_)$ increases exponentially. As $d_2(n)$ is finite, this process must eventually stabilise.
			
			Let $d,m_*$ be as constructed above and let $r$ then be as in \cref{prop:RExistenceDoubleIndexed}. We will demonstrate that these groups $\G$ satisfy \cref{prop:conjGap}.
			Using \cref{prop:inverseFactorialAproximation}, it suffices to show $\rf_{\G}(n)\simeq \frac{d_1(n)!}{2}$, and $\conj_{\G}(n)\simeq \frac{d_2(n)!}{2}.$
			The remainder of the proof of \cref{prop:conjGap} can thus be split into the following $4$ lemmas.

			\begin{lemma}
				$\mathrm{Rf}_{\G}(n)\prec \frac{d_1(n)!}{2}$.
			\end{lemma}
			\begin{proof}
				Let $w\in F_{a,b}$ be a freely reduced word such that $w(\alpha_*,\beta_*)$ is non trivial in $\G$. By \cref{prop:WordproblemIsLocal} we have that there exists some $n\leq2\norm{w}+1$ and some $m$ such that $w(\alpha_{d(n,m)},\beta_{r(n),d(n,m)})$ is non trivial.
				
				First suppose that $2\norm{w}+2\leq \frac{d(n,1)}{3}$. In this case, using \cref{prop:AltIsDInvariant}, we have that $w(\alpha_{d(n,1)},\beta_{r(n),d(n,m)})$ is non trivial. As $\alt(d(n,1))$ has order $d_1(n)!/2$, we obtain in this case that $\mathrm{Rf}_{\G}(w(\alpha_*,\beta_*))\leq d_1(2\norm{w}+1)!$.
				On the other hands suppose that $6\norm{w}+6> d(n,1)$. Either $d(n,m)<6\norm{w}+6$ for all $m$, or there exists some $m_0$ such that $d(n,m_0)\in [6\norm{w}+6,2C_6(6\norm{w}+6)]$. Again using \cref{prop:AltIsDInvariant}, we have that $w(\alpha_{d(n,m)},\beta_{r(n),d(n,m)})$ is non-trivial for some $m\leq m_0$.
				In this case we obtain that $\mathrm{Rf}_{\G}(w(\alpha_*,\beta_*))\leq (d(n,m_0))!/2$. Which can be bounded above by $\frac{d_1(2C_6(6\norm{w}+6))!}{2}$
			\end{proof}
			\begin{lemma}
				$\mathrm{Rf}_{\G}(n)\succ \frac{d_1(n)!}{2}$.
			\end{lemma}
			\begin{proof}
				The proof is analogous to \cite[Proposition 3.4.]{Bradford}.
				We have that $g_n=[\alpha^{r(n)}\beta\inv\alpha^{-r(n)},\beta\inv]$ is an element of $$
				\prod_{i=1}^{m_n}\alt(d(n,i))
				$$
				Let $\varphi:\G\rightarrow Q$ be a finite quotient, such that $\varphi(g)\neq 1$. As $\alt(d(n,i))$ is simple, we have that for any $i$, either $\varphi(\alt(d(n,i)))$ is trivial or that $\varphi$ is injective on $\alt(d(n,i))$. If the first is true for all $i\in[1,m_n]$, then $$\varphi(\prod_{i=1}^{m_n}\alt(d(n,i)))$$ is trivial and thus also $\varphi(g)$ is trivial. There thus exists some $i$ such that $\varphi$ is injective on $\alt(d(n,i))$ and thus is $\abs{Q}\geq\frac{d(n,i)!}{2}\geq\frac{d(n,1)!}{2}$. As $\norm{g_n}$ is at most $4r(n)+4\leq 148n+4$, we have that $\mathrm{Rf}_G^S(148n+4)\geq \frac{d(n,1)!}{2}$ as had to be shown.
			\end{proof}
			\begin{lemma}
				$\conj_{\G}(n)\succ \frac{d_2(n)!}{2}$
			\end{lemma}
			\begin{proof}
				Let $n$ be arbitrary, let $$g_1=\alpha_*^3\beta_*\inv[\alpha_*^{r(n)}\beta_*\inv \alpha_*^{-r(n)},\beta_*\inv]\beta_*\inv$$ and let $$g_2=\alpha_*^3\beta_*[\alpha_*^{r(n)}\beta_*\inv \alpha_*^{r(n)},\beta_*\inv]$$
				Notice that both of these elements have word norm at most $9+148n$.
				We will demonstrate that $$\conj(g_1,g_2)\geq d_2(n)!/2.$$
				
				First we will show that $g_1$ and $g_2$ are non-conjugate. For this it suffices to show that $g_1$ and $g_2$ are non-conjugate in some quotient. This is indeed the case as by \ref{prop:g1g2Conjugate} we have that $\pi_{n,m_n}(g_1)$ and $\pi_{n,m_n}(g_2)$ are non-conjugate. 
				
				We define $\tilde d$ and $\tilde m_*$ as follows:$$
				\tilde d(i,j)=\begin{cases}
					\begin{matrix}
						d(i,j)&\text{if }i\neq n\text{ or } j\leq m_n-1\\
						d(n,m-1)&\text{if } i=n\text{ and } j\geq m_n
					\end{matrix}
					\end{cases}
				$$
				and $$
				\tilde m_i=\begin{cases}
					\begin{matrix}
						m_i&\text{if }i\neq n\\
						m_n-1&\text{if }i=n
					\end{matrix}
				\end{cases}
				$$
				Notice that $\tilde d$ and $\tilde m_*$ satisfy the conditions \ref{cond:nondecreasing}, \ref{cond:primes}, \ref{cond:expGrowth} and \ref{cond:dsuflarge} of \cref{sec:constructingGroups}. Also notice that $\tilde d,\tilde m_*$ and $r$ also satisfy \ref{cond:RCong1},\ref{cond:RAproxN} and \ref{cond:RNonCongruent}.
				We can thus define the groups $G_0(\tilde d,\tilde m_*)$ and $G(\tilde d,\tilde m_*,r)$.
				It is clear that there exists a morphism $\pi:\Gn\rightarrow G_0(\tilde d,\tilde m_*)$ with as kernel $\alt(d(n,m))$. Furthermore, as $\pi$ maps the two generators of $\G$ to the two generators of $G(\tilde d,\tilde m_*,r)$, we have that this map restricts to an epimorphism $\pi:\G\rightarrow G(\tilde d,\tilde m_*,r)$.
				We will show that $\pi(g_1)$ and $\pi(g_2)$ are conjugate.

				We have that $\pi_{\infty}([\alpha^{r(n)}\beta\inv \alpha^{-r(n)},\beta\inv])$ vanishes, and by \ref{prop:ababVanishes} we have for $i\neq n$ that $\pi_{i,m}([\alpha^{r(n)}\beta\inv \alpha^{-r(n)},\beta\inv])$ vanishes. In particular, are $\pi_\infty(g_1)$ and $\pi_\infty(g_2)$, and are $\pi_{i,m}(g_1)$ and $\pi_{i,m}(g_2)$ equal and thus conjugate.
				Furthermore, since $d(n,m)\cong 1\mod 6$ for all $m\leq\tilde m_n$, it holds by \ref{prop:g1g2Conjugate} that $\pi_{n,m}(g_1)$ is conjugate to $\pi_{n,m}(g_1)$ for all such $m$.
				By \cref{prop:ConjIsLocal2} it follows that $\pi(g_1)$ and $\pi(g_2)$ are conjugate.
				
				Let now $\varphi:\G\rightarrow Q$ be a finite quotient such that $\varphi(g_1)$ and $\varphi(g_2)$ are non-conjugate. Let $K$ be the intersection of $\alt(d(n,m_n))$ and the kernel of $\varphi$. Suppose from contradiction that $K$ is non trivial. As $\alt(d(n,m_n))$ is simple, it follows that $K=\alt(d(n,m_n))$. This however implies that $\varphi$ factors through $\pi$. In particular, as $\pi(g_1)$ and $\pi(g_2)$ are conjugate, this would imply that $\varphi(g_1)$ and $\varphi(g_2)$ are conjugate, contradicting the construction of $\varphi$.
				
				It thus follows that $K$ is trivial. In particular is $\varphi$ injective on $\alt{d(n,m_n)}$ and thus is $\abs{Q}\geq d(n,m_n)!/2$. This holds for all such quotients $Q$ and thus must $\conj_\G(g_1,g_2)\geq d_2(n)!/2$ or thus $\conj_\G^S(9+148n)\geq d_2(n)!/2$.

			\end{proof}
			
			\begin{lemma}
				$\conj_{\G}(n)\prec d_2(n)!/2$
			\end{lemma}
			\begin{proof}
				This is completely analogous to the proof of \cref{prop:exactupperbound}, using \cref{prop:ConjIsLocal2} in stead of \cref{prop:ConjIsLocal}.
			\end{proof}
			
			\begin{appendices}
			\section{Conjugacy of $2$ elements}\label{sec:twoElements}
			In this section we demonstrate \cref{prop:g1g2Conjugate}.
			We demonstrate that if $r\cong d\cong 1\mod 6$, then  $\alpha_d^3\beta_{r,d}\inv[\alpha_d^r\beta_{r,d}\inv \alpha_d^{-r},\beta_{r,d}\inv]\beta_{r,d}\inv$and $\alpha_d^3\beta_{r,d}[\alpha_d^r\beta_{r,d}\inv \alpha_d^{-r},\beta_{r,d}\inv]$ are conjugate, but if $r\cong 1\mod 6, d\cong 5\mod 6$, then they are not conjugate.
			
			Let $g_1=\alpha_d^3\beta_{r,d}\inv[\alpha_d^r\beta_{r,d}\inv \alpha_d^{-r},\beta_{r,d}\inv]\beta_{r,d}\inv$ and $g_2=\alpha_d^3\beta_{r,d}[\alpha_d^r\beta_{r,d}\inv \alpha_d^{-r},\beta_{r,d}\inv]$.
			
			First the case where $r\cong d\cong 1\mod 6$. We first show that in this case both elements consist of a single cycle.
			notice that when $x\neq 0,r,2r,3r$, then for $g$ one of the elements of the above form, we have that $gx=x+3$ (or the reduction thereof $\mod d$ if $x\in\{d-3,d-2,d-1\})$. Given this information, a computation shows that $g_1$ consists of the one cycle $$\begin{pmatrix}
				0,3,6,&\cdots,& 3r,\\
				2r+3,2r+6,&\cdots,& d-2,\\
				1,4,&\cdots,&r,\\
				3r+3,3r+6,&\cdots,&d-1,\\
				2,5,&\cdots,&2r,\\
				r+3,r+6,&\cdots,&d-3
			\end{pmatrix}$$
			As similar computation shows that the cycle structure of $g_2$ is given by
			$$\begin{pmatrix}
				0,3r+3,3r+6,&\cdots,&d-1,\\
				2,5,&\cdots,&2r,\\
				2r+3,2r+6,&\cdots,&d-2,\\
				1,4,&\cdots,&r,\\
				3,6,&\cdots,&3r,\\
				r+3,r+6,&\cdots,&d-3
			\end{pmatrix}$$
			We thus obtain that $g_1$ and 
			$g_2$ have the same cycle structure. We construct a conjugator $h$ between $g_1$ and $g_2$ as follows: Let $h\in \mathrm{Sym}(d)$ be the permutation such that $h(0)=0$ and such that $h\comp g_1(x)=g_2\comp h(x)$. As $g_1$ and $g_2$ only have a single cycle, (and thus in particular these cycles have the same length), this uniquely defines $h$. We have that $h$ is the following map:\begin{multicols}{3}
				$0\mapsto0$\\$3\mapsto 3r+3$\\$6\mapsto 3r+6$\\\vdots\\$3r\mapsto6r$\\$2r+3\mapsto6r+3$\\\vdots\\$d-4r-1\mapsto d-1$\\$d-4r+2\mapsto 2$\\\vdots\\$d-2r\mapsto 2r$\\$d-2r+3\mapsto2r+3$\\$d-2\mapsto 4r-2$\\$1\mapsto 4r+1$\\\vdots\\$r\mapsto 5r$\\$3r+3\mapsto 5r+3$\\\vdots\\$d-2r-2\mapsto d-2$\\$d-2r+1\mapsto 1$\\\vdots\\$d-r\mapsto r$\\$d-r+3\mapsto 3$\\\vdots\\$d-1\mapsto r-1$\\$2\mapsto r+2$\\\vdots\\$2r\mapsto 3r$\\$r+3\mapsto r+3$\\\vdots\\$d-3\mapsto d-3.$			
			\end{multicols}
			In case of dots, $3$ is added repeatedly to both the argument and function value.
			We will show that $h$ as above is an even permutation, and thus that $h\in \alt(d)$. For this we want to compute the parity of the number of inversions. That is the number of pairs $x,y$ such that $x\prec y$ but $h(x)\succ h(y)$ for some order $\prec$.
			As order we pick the order in which elements appear in the cycle of $g_1$, where we begin with $0$. That is $$0\prec 3\prec 6\prec\cdots\prec d-2\prec 1\prec\cdots\prec d-3.$$
			To determine the number of inversions, we subdivide into blocks as follows:$$
			\begin{matrix}
				\{0\}&(0)\\
				\{3,6,\cdots 3r\}&(1)\\
				\{2r+3,2r+6,\cdots,d-4r-1,d-4r+2,\cdots,d-2r\}&(2)\\
				\{d-2r+3\cdots d-2,1,\cdots,r\}&(3)\\
				\{3r+3,\cdots d-r\}&(4)\\
				\{d-r+3,\cdots,d-1,2,\cdots,2r\}&(5)\\
				\{r+3,\cdots,d-3\}&(6)
			\end{matrix}$$
			These blocks each consist out a sequence of consecutive elements, and $h$ maps them again to a sequence of consecutive elements while preserving the order within blocks.
			As $h$ preserves the order within blocks, we have that if $h$ inverts $x$ and $y$, then $x$ and $y$ must lie in distinct blocks.
			Furthermore, as blocks get mapped to consecutive elements, we have that if $h$ inverts $x$ and $y$, and if $x'$ and $y'$ lie in the same block as $x$ and $y$ respectively, then also $x',y'$ is an inversion.
			
			Notice that all blocks above have odd length, so if two blocks invert, then there must be an odd number of inversions between those blocks.			
			As the number of inversions between the blocks is $8$, an even number, we also have that the number of inversions of $h$ must be even, or thus that $h$ is an even permutation and an element of $\alt(d)$.
			
			As $hg_1h\inv=g_2$ we have in this case that $g_1$ and $g_2$ are conjugate in $\alt(d)$.
			
			Now consider the case where	$r\cong 1\mod 6, d\cong 5\mod 6$. We show that in this case $g_1$ and $g_2$ have distinct cycle structure and thus that they are distinct.
			We compute the cycle structure of $g_1$ to be $$\begin{pmatrix}
				0,3,&\cdots,&3r,\\
				2r+3,&\cdots,&d-3
			\end{pmatrix}
			\begin{pmatrix}
				1,4,&\cdots,&, r,\\
				3r+3,&\cdots,&d-2
			\end{pmatrix}
			\begin{pmatrix}
				2,5,&\cdots&,2r,\\
				r+3,&\cdots,&d-1
			\end{pmatrix}
			$$
			and the cycle structure of $g_2$ to be $$\begin{pmatrix}
				0,3r+3,3r+6,&\cdots,&d-2,\\
				1,4,&\cdots,&r,\\
				3,6,&\cdots,&3r,\\
				r+3,r+6,&\cdots,&d-1,\\
				2,5,&\cdots,&2r,\\
				2r+3,&\cdots,&d-3
			\end{pmatrix}
			$$
			The number of cycles is distinct and thus are $g_1$ and $g_2$ non-conjugate.
			
			\end{appendices}
		\bibliographystyle{plain}
		\bibliography{bibliography}
	
\end{document}